\documentclass[preprint]{imsart}
\RequirePackage[OT1]{fontenc} \RequirePackage{amsthm,amsmath}
\RequirePackage[numbers]{natbib}
\RequirePackage[colorlinks,citecolor=blue,urlcolor=blue]{hyperref}
\RequirePackage{hypernat}
\usepackage{amsmath}
\usepackage{amssymb}
\usepackage{amsmath}
\usepackage{amsfonts}
\usepackage[american]{babel}
\usepackage{graphicx}
\usepackage{flafter}
\usepackage[section]{placeins}

\startlocaldefs

\def\E{{\mathbb E }}

\def\eps{{\varepsilon}}

\def\Bbb E{\mathbb{E}}
\def\Bbb R{\mathbb{R}}
\newtheorem{lemma}{Lemma}
\newtheorem{theorem}{Theorem}
\newtheorem{definition}{Definition}
\newtheorem{corollary}{Corollary}

\makeatletter \@addtoreset{equation}{section}

\makeatother

\font\tencmmib=cmmib10 \skewchar\tencmmib '60
\newfam\cmmibfam
\textfont\cmmibfam=\tencmmib

\font\tenmsb=msbm10 
\def\Bbb#1{\hbox{\tenmsb#1}}

\def\bbox{\quad\hbox{\vrule \vbox{\hrule \vskip2pt \hbox{\hskip2pt
\vbox{\hsize=1pt}\hskip2pt} \vskip2pt\hrule}\vrule}}
\def\lessim{\ \lower4pt\hbox{$
\buildrel{\displaystyle <}\over\sim$}\ }
\def\gessim{\ \lower4pt\hbox{$\buildrel{\displaystyle >}
\over\sim$}\ }

\def\eps{\varepsilon}

\def\go0{\to 0}

\def\leftitem#1{\item{\hbox to\parindent{\enspace#1\hfill}}}

\def\qed{{$\hfill \bbox$}}

\def\sg{\sigma}

\def\sg2{\sigma^2}

\def\__{_{\infty}}

\numberwithin{equation}{section} \theoremstyle{plain}

\newcommand{\1}{{\rm 1}\kern-0.24em{\rm I}}
\def\E{\mathbb E}

\endlocaldefs

\begin{document}

\begin{frontmatter}
\title{Concentration Inequalities and Moment Bounds for Sample Covariance Operators} \runtitle{Concentration of sample covariance}

\begin{aug}
\author{\fnms{Vladimir} \snm{Koltchinskii}\thanksref{t1}\ead[label=e1]{vlad@math.gatech.edu}} and 
\author{\fnms{Karim} \snm{Lounici}\thanksref{m1}\ead[label=e2]{klounici@math.gatech.edu}}
\thankstext{t1}{Supported in part by NSF Grants DMS-1207808, CCF-0808863 and CCF-1415498}
\thankstext{m1}{Supported in part by NSF Grant DMS-11-06644 and by Simons Foundation Grant 315477}
\runauthor{V. Koltchinskii and K. Lounici}

\affiliation{Georgia Institute of Technology\thanksmark{m1}}

\address{School of Mathematics\\
Georgia Institute of Technology\\
Atlanta, GA 30332-0160\\
\printead{e1}\\
\printead*{e2}
}
\end{aug}

\begin{abstract}
Let $X,X_1,\dots, X_n,\dots$ be i.i.d. centered Gaussian random variables in a separable 
Banach space $E$ with covariance operator $\Sigma:$
$$
\Sigma:E^{\ast}\mapsto E,\ \ \Sigma u = {\mathbb E}\langle X,u\rangle, u\in E^{\ast}.
$$
The sample covariance operator $\hat \Sigma:E^{\ast}\mapsto E$ is defined as 
$$
\hat \Sigma u := n^{-1}\sum_{j=1}^n \langle X_j,u\rangle X_j, u\in E^{\ast}.
$$
The goal of the paper is to obtain concentration inequalities and expectation bounds for the 
operator norm $\|\hat \Sigma-\Sigma\|$ of the deviation of the sample covariance operator 
from the true covariance operator. In particular, it is shown that 
$$
{\mathbb E}\|\hat \Sigma-\Sigma\|\asymp 
\|\Sigma\|\biggl(\sqrt{\frac{{\bf r}(\Sigma)}{n}}\bigvee \frac{{\bf r}(\Sigma)}{n}\biggr),
$$
where 
$$
{\bf r}(\Sigma):=\frac{\Bigl({\mathbb E}\|X\|\Bigr)^2}{\|\Sigma\|}.
$$
Moreover, it is proved that,  under the assumption that ${\bf r}(\Sigma)\leq n,$    
for all $t\geq 1,$
with probability at least $1-e^{-t}$
\begin{align}
&
\nonumber
\Bigl|\|\hat\Sigma - \Sigma\|-M\Bigr| 
\lesssim \|\Sigma\|\biggl(\sqrt{\frac{t}{n}}\bigvee \frac{t}{n}\biggr),
\end{align}
where $M$ is either the median, or the expectation of $\|\hat \Sigma-\Sigma\|.$ 
On the other hand, under the assumption that ${\bf r}(\Sigma)\geq n,$ 
for all $t\geq 1,$ with probability at least $1-e^{-t}$
\begin{align}
&
\nonumber
\Bigl|\|\hat\Sigma - \Sigma\|-M\Bigr| 
\lesssim \|\Sigma\|\biggl(\sqrt{\frac{{\bf r}(\Sigma)}{n}}\sqrt{\frac{t}{n}}\bigvee \frac{t}{n}\biggr). 
\end{align}
\end{abstract}



\end{frontmatter}
%
%
%



\section{Introduction}

Let $(E,\|\cdot\|)$ be a separable Banach space with the dual space 
$E^{\ast}.$ For $x\in E, u\in E^{\ast},$ $\langle x,u\rangle$ denotes 
the value of linear functional $u$ at vector $x.$ 
Let $X$ be a centered random variable in $E$ with  
${\mathbb E}|\langle X,u\rangle|^2<+\infty, u\in E^{\ast}$ (that is, $X$ is weakly 
square integrable). 
Let  
$$
\Sigma u := {\mathbb E}\langle X,u\rangle X,\ u\in E^{\ast}.
$$ 
It is well known that this defines a bounded symmetric nonnegatively definite operator 
$\Sigma : E^{\ast}\mapsto E$ that is called \it the covariance operator \rm 
of random variable $X.$ Moreover, if ${\mathbb E}\|X\|^2<+\infty$ (so, $X$ is strongly 
square integrable), then it is also well known that the covariance operator $\Sigma$ is nuclear. Recall that a linear operator $A$ from a Banach space $E_1$ into a 
Banach space $E_2$ is called nuclear iff  
there exist sequences $\{x_n:n\geq 1\}\subset E_1^{\ast},\ \{y_n:n\geq 1\}\subset E_2$ such 
that 
\begin{equation}
\label{nuce_1}
A u = \sum_{n\geq 1} \langle u, x_n\rangle y_n, u\in E_1 
\end{equation}
and 
\begin{equation}
\label{nuce_2}
\sum_{n\geq 1}\|x_n\|\|y_n\|<\infty.
\end{equation}
The nuclear norm $\|A\|_1$ is defined as the infimum of the sums (\ref{nuce_2})
over all the sequences $\{x_n:n\geq 1\}\subset E_1^{\ast},\ \{y_n:n\geq 1\}\subset E_2$ such that 
representation (\ref{nuce_1}) holds. 

Let $X_1,\dots,X_n$
be i.i.d. copies of $X.$ The sample (empirical) covariance operator based on the observations $(X_1,\dots, X_n)$ is defined as the operator $\hat \Sigma: E^{\ast}\mapsto E$ such that
$$
\hat \Sigma u :=n^{-1}\sum_{j=1}^n \langle X_j,u\rangle X_j,\ u\in E^{\ast}.
$$
Clearly, this is an operator of rank at most $n$ and it is an unbiased estimator 
of the covariance operator $\Sigma .$

In this paper, we are interested in the case when $X$ is a centered Gaussian random 
vector in $E$ with covariance $\Sigma.$ This implies that ${\mathbb E}\|X\|^2<+\infty$ (in fact, $\|X\|$ is even a random variable with a finite $\psi_2$-norm, see \cite{Ledoux}, Chapter 3) and, as a consequence, the covariance operator $\Sigma$ is nuclear. For operators $A: E^{\ast}\mapsto E,$ $\|A\|$
will denote the operator norm: 
$$
\|A\|:=
\sup_{u\in E^{\ast}, \|u\|\leq 1} \|Au\|
= \sup_{u,v\in E^{\ast}, \|u\|\leq 1, \|v\|\leq 1} \Bigl|\langle Au,v\rangle\Bigr|.
$$

Several other definitions and notations will be used throughout the paper. In particular,
the relationship $B_1\lesssim B_2$ (for nonnegative $B_1, B_2$) means that there exists an absolute constant $c\in (0,\infty)$ such that $B_1\leq cB_2.$ Similarly, $B_1 \gtrsim B_2$ means that $B_1\geq cB_2$ for an absolute constant $c.$ If both $B_1\lesssim B_2$ and $B_1\gtrsim B_2,$ we write 
$B_1\asymp B_2.$ Sometimes, symbols $\lesssim, \gtrsim, \asymp$ are provided with subscripts 
indicating possible dependence of constant $c$ on other constants (say, $B_1\lesssim_{a} B_2$
would mean that $B_1\leq cB_2$ with $c$ that might depend on $a$). 

We will also use occasionally Orlicz norms (such as $\psi_1$- and $\psi_2$-norms) in the 
spaces of random variables. Given a convex nondecreasing function $\psi:{\mathbb R}_+\mapsto {\mathbb R}_+$ with $\psi(0)=0$ and a random variable $\eta$ on a probability space 
$(\Omega,{\mathcal A}, {\mathbb P}),$ define its $\psi$-norm as 
$$
\|\eta\|_{\psi}:=\inf \biggl\{C>0: {\mathbb E}\psi\biggl(\frac{|\eta|}{C}\biggr)\leq 1\biggr\}.
$$  
For $\psi (u):= u^p,u>0,$ $p\geq 1,$ the $\psi$-norm coincides with the $L_p({\mathbb P})$-norm. 
Consider also $\psi_2(u):= e^{u^2}-1, u\geq 0$ and $\psi_1(u)=e^u-1, u\geq 0.$
Then $\|\eta\|_{\psi_2}<+\infty$ means that $\eta$ has subgaussian tails and $\|\eta\|_{\psi_1}<+\infty$ means that $\eta$ has subexponential tails.
Some well known inequalities for $\psi_1$ random variables will be used in 
what follows. For instance, for arbitrary random variables $\xi_k, k=1,\dots, N, N\geq 2$
with $\|\xi_k\|_{\psi_1}<+\infty,$
$$
{\mathbb E}\max_{1\leq k\leq N}|\xi_k|\lesssim \max_{1\leq k\leq N}\|\xi_k\|_{\psi_1}\log N.
$$
If $\xi, \xi_1,\dots, \xi_n$ are i.i.d. centered random variables with $\|\xi\|_{\psi_1}<+\infty,$ then the sum $\xi_1+\dots+\xi_n$ satisfies the 
following version of Bernstein's inequality: for all $t\geq 0$ with probability 
at least $1-e^{-t}$
$$
\biggl|\frac{\xi_1+\dots+\xi_n}{n}\biggr|\lesssim \|\xi\|_{\psi_1}\biggl(\sqrt{\frac{t}{n}}\bigvee \frac{t}{n}\biggr). 
$$

Our goal is to obtain moment bounds and concentration inequalities for the operator norm $\|\hat \Sigma-\Sigma\|.$ It turns out that both the size of the expectation of random variable $\|\hat \Sigma-\Sigma\|$ and its concentration around its mean can be characterized 
in terms of the operator norm $\|\Sigma\|$ and another parameter defined below.

\begin{definition}
Assuming that $X$ is a centered Gaussian random variable in $E$ with covariance 
operator $\Sigma,$ define
$$
{\bf r}(\Sigma):=\frac{\Bigl({\mathbb E}\|X\|\Bigr)^2}{\|\Sigma\|}.
$$
\end{definition}

Note that, for a Gaussian vector $X,$ ${\mathbb E}^{1/2}\|X\|^2\asymp {\mathbb E}\|X\|$ implying that 
$$
{\bf r}(\Sigma)\leq \frac{{\mathbb E}\|X\|^2}{\|\Sigma\|}=:\tilde {\bf r}(\Sigma)\lesssim {\bf r}(\Sigma).
$$ 
In the case when $E$ is a Hilbert space, ${\mathbb E}\|X\|^2={\rm tr}(\Sigma)$ and 
$$
\tilde {\bf r}(\Sigma)=\frac{{\rm tr}(\Sigma)}{\|\Sigma\|}.
$$ 
The last quantity has been already used in the literature under the name of ``effective rank''(see \cite{Vershynin}). 
Clearly, $\tilde {\bf r}(\Sigma)\leq {\rm rank}(\Sigma).$

The main results of the paper include the following:

\begin{itemize} 
\item under an assumption that $X,X_1,\dots, X_n$ are i.i.d. centered 
Gaussian random variables in $E$ with covariance operator $\Sigma,$ it 
will be shown that 
\begin{equation}
\label{main_bd_1}
{\mathbb E}\|\hat \Sigma-\Sigma\|\asymp 
\|\Sigma\|\biggl(\sqrt{\frac{{\bf r}(\Sigma)}{n}}\bigvee \frac{{\bf r}(\Sigma)}{n}\biggr).
\end{equation}
\item Moreover, under an additional assumption that ${\bf r}(\Sigma)\lesssim n,$ the following concentration 
inequality holds for some constant $C>0$ and for all $t\geq 1$ with probability at least  
$1-e^{-t}:$
\begin{equation}
\label{main_bd_2}
\Bigl|\|\hat\Sigma - \Sigma\|-{\mathbb E}\|\hat\Sigma - \Sigma\|\Bigr| 
\leq C\|\Sigma\|\biggl(\sqrt{\frac{t}{n}}\bigvee \frac{t}{n}\biggr).
\end{equation}
Under an assumption that ${\bf r}(\Sigma)\gtrsim n,$ the concentration inequality 
becomes 
\begin{equation}
\label{main_bd_2''}
\Bigl|\|\hat\Sigma - \Sigma\|-{\mathbb E}\|\hat\Sigma - \Sigma\|\Bigr| 
\leq C\|\Sigma\|\biggl(\sqrt{\frac{{\mathbf r}(\Sigma)}{n}}\sqrt{\frac{t}{n}}\bigvee \frac{t}{n}\biggr) 
\end{equation}
and it holds with the same probability.
\end{itemize}

\section{Main results}\label{Sec:Prelim}



The problem of bounding the operator norm $\|\hat \Sigma-\Sigma\|$ has been intensively 
studied, especially, in the finite-dimensional case (see \cite{Vershynin} and references
therein). The focus has been on understanding 
of dependence of this norm on the dimension of the space and on the sample size $n$ (that 
could be simultaneously large) as well as on the tails of linear forms $\langle X,u\rangle, u\in E$ and of the norm $\|X\|$ of random variable $X.$ Many results that hold for Gaussian 
random variables are also true in a slightly more general subgaussian case.

\begin{definition}
A centered random variable $X$ in $E$ will be called {\it subgaussian} iff, for all $u\in E^{\ast},$ 
$$
\|\langle X,u\rangle\|_{\psi_2}\lesssim \|\langle X,u\rangle\|_{L_2({\mathbb P})}.
$$ 
\end{definition}

We will also need the following definition.

\begin{definition}
A weakly square integrable centered random variable $X$ in $E$ with covariance operator $\Sigma$ is called {\it pregaussian} iff there exists a centered Gaussian random variable $Y$
in $E$ with the same covariance operator $\Sigma.$  
\end{definition}

Suppose now that $E={\mathbb R}^d$ for some $d\geq 1.$ 
It will be viewed as a standard Euclidean space. 
Then, the following result is well known (it is a slight modification 
of Theorem 5.39 in Vershynin \cite{Vershynin} stated there for isotropic subgaussian random 
variables, that is, when $\Sigma$ is the identity operator).  

\begin{theorem}
\label{classical}
There exists an absolute constant $C>0$ such that, for all $t\geq 1,$
with probability at least $1-e^{-t}$ 
$$
\|\hat \Sigma - \Sigma\| \leq C \|\Sigma\|\left(\sqrt{\frac{d}{n}}\bigvee \frac{d}{n}\bigvee \sqrt{\frac{t}{n}}\bigvee \frac{t}{n} \right).
$$ 
\end{theorem}

The proof of this theorem is based on a simple $\eps$-net argument that allows one to reduce bounding 
the operator norm $\|\hat \Sigma - \Sigma\|$ to bounding the finite maximum 
$$
\max_{u\in M}|\langle (\hat \Sigma -\Sigma) u,u\rangle|=
\max_{u\in M}\biggl|n^{-1}\sum_{j=1}^n \langle X_j,u\rangle^2 -{\mathbb E}\langle X,u\rangle^2\biggr|,
$$
where $M\subset S^{d-1}$ is a $1/4$-net of the unit sphere 
of cardinality ${\rm card}(M)\leq 9^d.$ The bounding of the finite maximum 
is based on a version of Bernstein inequality for the sum of independent $\psi_1$ random variables $\langle X_j,u\rangle^2$ combined with the union bound 
(see the proof of Theorem 5.39 in \cite{Vershynin} and the comments after this theorem).

In the isotropic case (that is, when $\Sigma=I_d$), the bound of Theorem \ref{classical} 
is sharp and it can be viewed as a non-asymptotic version of the well known Bai-Yin 
theorem from the asymptotic theory of random matrices. In the cases when the distribution 
of $X$ is far from being isotropic, this bound is no longer sharp and it clearly can not 
be used in the infinite-dimensional case.
If the covariance operator $\Sigma$ is of a small rank, it is natural to expect that 
the rank of $\Sigma$ rather than the dimension of the space $E$ should be involved 
in the bound. It turns out that one can obtain bounds on the operator norm 
$\|\hat \Sigma-\Sigma\|$ in terms of the ``effective rank'' 
$\tilde {\bf r}(\Sigma)=\frac{{\rm tr}(\Sigma)}{\|\Sigma\|}$ of the 
covariance operator $\Sigma$ (that is always dominated by its actual rank). 
This could be done, for instance, using noncommutative Bernstein type inequalities 
that go back to Ahlswede and Winter \cite{Ahlswede} (see also Tropp \cite{Tropp}, Koltchinskii \cite{Koltchinskii}).  
For instance, Lounici \cite{Lounici} showed that with some constant $C>0$ and with probability at least $1-e^{-t}$ 
$$
\|\hat \Sigma - \Sigma\| \leq C\|\Sigma\| 
\max\left\{ \sqrt{\frac{{\tilde {\bf r}}(\Sigma)\log d + t}{n}} , \frac{({\tilde {\bf r}}(\Sigma)\log d + t)\log n}{n} \right\}.
$$

Another approach to bounding the operator norm $\|\hat \Sigma-\Sigma\|$ was developed 
by Rudelson \cite{Rudelson} and it is based on 
a noncommutative Khintchine inequality due to Lust-Picard and Pisier \cite{Pisier}. 
This method can be used not only in subgaussian, but also in ``heavy tailed'' cases
and it leads, for instance, to the following expectation bound (see Vershynin \cite{Vershynin}, Theorem 5.48):
$$
{\mathbb E}\|\hat \Sigma-\Sigma\|\lesssim \max\left\{
\|\Sigma\|^{1/2}{\mathbb E}^{1/2}\max_{1\leq j\leq n}\|X_j\|^2
\sqrt{\frac{\log d}{n}},{\mathbb E}\max_{1\leq j\leq n}\|X_j\|^2\frac{\log d}{n}
\right\}.
$$ 
Note that, in the subgaussian case, 
$$
\Bigl\|\|X\|^2\Bigr\|_{\psi_1} \lessim {\rm tr}(\Sigma),
$$ 
which implies that
$$
{\mathbb E}\max_{1\leq j\leq n}\|X_j\|^2\lessim {\rm tr}(\Sigma) \log n
=\|\Sigma\|\tilde{{\bf r}}(\Sigma)\log n. 
$$
Therefore, in this case we get 
$$
{\mathbb E}\|\hat \Sigma-\Sigma\|\lesssim \|\Sigma\|\max\left\{
\sqrt{\frac{\tilde{{\bf r}}(\Sigma)\log d \log n}{n}},
\frac{\tilde{{\bf r}}(\Sigma)\log d \log n}{n}
\right\}.
$$ 

In each of the above approaches, the bounds are not dimension free (at least, with a straightforward application of noncommutative Bernstein or Khintchine inequalities)
and they could not be directly used in the infinite-dimensional case. We will use below 
a different approach based on recent deep results on generic chaining bounds for 
empirical processes.
The following facts about generic chaining complexities will be needed. 
Let $N_n:=2^{2^n}, n\geq 1$ and $N_0:=1.$ 
Given a metric space $(T,d),$ an increasing sequence $\Delta_n$ of partitions 
of $T$ is called admissible if ${\rm card}(\Delta_n)\leq N_n.$ For $t\in T,$
$\Delta_n(t)$ denotes the unique set of the partition $\Delta_n$ that contains $t.$
For $A\subset T,$ $D(A)$ denotes the diameter of set $A.$
Define  
$$
\gamma_2(T,d) = \inf\sup_{t\in T} \sum_{n=0}^\infty 2^{n/2}D(\Delta_n(t)),
$$
where the infimum is taken over all admissible sequences. 

The following fundamental result is due to Talagrand (see \cite{Talagrand}; it was initially 
stated it terms of majorizing measures rather than generic complexities). 

\begin{theorem}
\label{tal}
Let $X(t), t\in T$ be a centered Gaussian process and suppose that 
$$
d(t,s):= {\mathbb E}^{1/2}(X(t)-X(s))^2, t,s\in T.
$$
Then, there exists an absolute constant $K>0$ such that 
$$
{\mathbb E}\sup_{t\in T}X(t)\geq K^{-1} \gamma_2(T;d).
$$
\end{theorem}

In what follows, generic chaining complexities are used in the case when $T={\mathcal F}$ is a function class on a probability space $(S,{\mathcal A},P)$ and $d$ is the metric generated by either $L_2(P)$-norm, or by the $\psi_2$-norm with respect to $P.$ 
We will use the following result due to Mendelson \cite{Mendelson}
(although an earlier, simpler and weaker version, with $\sup_{f\in {\cal F}}\|f\|_{\psi_2}$ 
instead of $\sup_{f\in {\cal F}}\|f\|_{\psi_1},$ that goes back to Klartag and Mendelson 
\cite{Klartag} would suffice for our purposes).  

\begin{theorem}
\label{men}
Let $X,X_1,\dots, X_n$ be i.i.d. random variables in $S$ with common distribution 
$P$ and let ${\cal F}$ be a class of measurable functions on $(S,{\mathcal A})$
such that $f\in {\cal F}$ implies $-f\in {\cal F}$ and ${\mathbb E}f(X)=0.$ Then
$$
{\mathbb E}\sup_{f\in {\mathcal F}} 
\left| \frac{1}{n}\sum_{i=1}^n f^2(X_i)-   \mathbb E f^2(X) \right|
\lesssim \max\left\lbrace\sup_{f\in {\mathcal F}}\|f\|_{\psi_1} 
\frac{\gamma_2({\mathcal F};\psi_2)}{\sqrt{n}},
\frac{\gamma_2^2({\mathcal F};\psi_2)}{n}\right\rbrace.
$$
\end{theorem}


Assume again that $E$ is an arbitrary separable Banach space. 
The next result provides a characterization of the size of ${\mathbb E}\|\hat \Sigma-\Sigma\|$
in terms of the parameters $\|\Sigma\|$ and ${\bf r}(\Sigma)$ for Gaussian random variable $X$
(the upper bound also holds in the case when $X$ is both subgaussian and pregaussian). 

\begin{theorem}
\label{th_operator}
Let $X,X_1,\ldots,X_n$ be i.i.d. weakly square integrable centered random vectors in $E$ with covariance operator $\Sigma.$ If $X$ is subgaussian and pregaussian, then 
\begin{align}
\nonumber 
\E\|\hat\Sigma - \Sigma\| \lesssim 
\|\Sigma\|\max\left\lbrace \sqrt{\frac{\mathbf{r}(\Sigma)}{n}}, \frac{\mathbf{r}(\Sigma)}{n}\right\rbrace.
\end{align}
Moreover, if $X$ is Gaussian, then 
\begin{align}
\nonumber
\|\Sigma\|\max\left\lbrace \sqrt{\frac{\mathbf{r}(\Sigma)}{n}}, \frac{\mathbf{r}(\Sigma)}{n}\right\rbrace \lesssim 
\E\|\hat\Sigma - \Sigma\| \lesssim 
\|\Sigma\|\max\left\lbrace \sqrt{\frac{\mathbf{r}(\Sigma)}{n}}, \frac{\mathbf{r}(\Sigma)}{n}\right\rbrace.
\end{align}
\end{theorem}

\begin{proof}
The proof of the upper bound relies 
on the generic chaining bound of Theorem \ref{men}, while the proof of the lower bound is rather elementary.

{\it Upper bound}. We have
\begin{align*}
&
{\mathbb E}\|\hat \Sigma - \Sigma\| =
{\mathbb E}\sup_{\|u\|\leq 1, \|v\|\leq 1}\langle (\hat \Sigma-\Sigma)u,v\rangle
\\
&
=
{\mathbb E}\sup_{\|u\|\leq 1, \|v\|\leq 1}
\left(
\left\langle (\hat \Sigma-\Sigma)\frac{u+v}{2},\frac{u+v}{2}\right\rangle
-
\left\langle (\hat \Sigma-\Sigma)\frac{u-v}{2},\frac{u-v}{2}\right\rangle
\right)
\\
&
\leq 2\sup_{\|u\|\leq 1}\Bigl|\langle (\hat \Sigma-\Sigma)u,u\rangle\Bigr|
= 2{\mathbb E}\sup_{\|u\|\leq 1} \left| \frac{1}{n}\sum_{i=1}^n \left\langle  X_i,u\right\rangle^2 -   
\left\langle  \Sigma u, u\right\rangle \right|
\\
&
= 
2{\mathbb E}\sup_{f\in {\mathcal F}} 
\left| \frac{1}{n}\sum_{i=1}^n f^2(X_i)-   \mathbb E f^2(X) \right|,
\end{align*}
where ${\mathcal F}:=\Bigl\{\langle \cdot,u\rangle: u\in U_{E^{\ast}}\Bigr\},$
$U_{E^{\ast}}:=\{u\in E^{\ast}:\|u\|\leq 1\}$ 
and $P$ is the distribution of random variable $X.$

Since $X$ is subgaussian, the $\psi_1$- and $\psi_2$-norms of linear functionals 
$\langle X,u\rangle$ are both equivalent to the $L_2$-norm. This implies 
that 
$$
\sup_{f\in {\mathcal F}}\|f\|_{\psi_1}\lesssim \sup_{u\in U_{E^{\ast}}}{\mathbb E}^{1/2}\langle
X,u\rangle^2 \leq \|\Sigma\|^{1/2}.
$$
Also, since $X$ is pregaussian, there exists a centered Gaussian random variable $Y$
in $E$ with the same covariance $\Sigma.$ This means that 
$$
d_{Y}(u,v)=\|\langle \cdot,u\rangle-\langle \cdot,v\rangle\|_{L_2(P)}, u,v\in U_{E^{\ast}}.
$$ 
Using Talagrand's Theorem \ref{tal}, we easily get that
$$
\gamma_2({\mathcal F},\psi_2) \lesssim 
\gamma_2({\mathcal F},L_2)=\gamma_2(U_{E^{\ast}};d_Y)\lesssim 
\mathbb E \sup_{u\in U_{E^{\ast}}} 
\langle Y, u\rangle \leq {\mathbb E}\|Y\|.
$$
Therefore, it follows that 
$$
{\mathbb E}\|\hat \Sigma - \Sigma\|  
\lesssim 
\max\left\lbrace \|\Sigma\|^{1/2}\frac{{\mathbb E}\|Y\|}{\sqrt{n}}, 
\frac{({\mathbb E}\|Y\|)^2}{n}\right\rbrace
\lesssim 
\|\Sigma\|\max\left\lbrace \sqrt{\frac{\mathbf{r}(\Sigma)}{n}}, \frac{\mathbf{r}(\Sigma)}{n}\right\rbrace,
$$
which proves the upper bound. 

{\it Lower Bound.} To prove the lower bound, note that 
\begin{align}
\label{odin}
&
\nonumber
{\mathbb E}\|\hat \Sigma-\Sigma\|=
{\mathbb E}\sup_{\|u\|\leq 1}\biggl\|n^{-1}\sum_{j=1}^n 
\langle X_j,u\rangle X_j-{\mathbb E}\langle X,u\rangle X\biggr\|
\\
&
\geq 
\sup_{\|u\|\leq 1}{\mathbb E}\biggl\|n^{-1}\sum_{j=1}^n 
\langle X_j,u\rangle X_j-{\mathbb E}\langle X,u\rangle X\biggr\|.
\end{align}
For a fixed $u\in E^{\ast}$ with $\|u\|\leq 1$ and $\langle \Sigma u,u\rangle>0,$
denote 
$$
X' := X- \langle X,u\rangle \frac{\Sigma u}{\langle \Sigma u,u\rangle}.
$$
By a straightforward computation, for all $v\in E^{\ast},$ the random variables 
$\langle X,u\rangle$ and $\langle X',v\rangle$ are uncorrelated. Since they are jointly
Gaussian, it follows that $\langle X,u\rangle$ and $X'$ are independent. 
Define 
$$
X_j' := X_j- \langle X_j,u\rangle \frac{\Sigma u}{\langle \Sigma u,u\rangle}, j=1,\dots, n.
$$
Then $\{X_j':j=1,\dots,n\}$ and $\{\langle X_j,u\rangle:j=1,\dots, n\}$ are also 
independent. We easily get  
\begin{align}
\label{dva}
&
\nonumber
{\mathbb E}\biggl\|n^{-1}\sum_{j=1}^n 
\langle X_j,u\rangle X_j-{\mathbb E}\langle X,u\rangle X\biggr\|
=
\\
&
{\mathbb E}\biggl\|n^{-1}\sum_{j=1}^n 
(\langle X_j,u\rangle^2 - 
{\mathbb E}\langle X,u\rangle^2)\frac{\Sigma u}{\langle \Sigma u,u\rangle}+
n^{-1}\sum_{j=1}^n \langle X_j,u\rangle X_j'\biggr\|,
\end{align}
where we used the fact that 
$$
{\mathbb E}\langle X,u\rangle X={\mathbb E}\langle X,u\rangle^2
\frac{\Sigma u}{\langle \Sigma u,u\rangle}+ {\mathbb E}\langle X,u\rangle {\mathbb E}X'
={\mathbb E}\langle X,u\rangle^2
\frac{\Sigma u}{\langle \Sigma u,u\rangle}.
$$
Note that, conditionally on $\langle X_j,u\rangle, j=1,\dots, n,$
the distribution of random variable 
$$n^{-1}\sum_{j=1}^n \langle X_j,u\rangle X_j'$$
is Gaussian and it coincides with the distribution of the random variable 
$$
\biggl(n^{-1}\sum_{j=1}^n \langle X_j,u\rangle^2\biggr)^{1/2}\frac{X'}{\sqrt{n}}.
$$ 
Denote now by ${\mathbb E}_u$ the conditional expectation given $\langle X_j,u\rangle, j=1,\dots, n$ and by ${\mathbb E}'$ the conditional expectation given $X_1',\dots, X_n'.$
Then, we have 
$$
{\mathbb E}\biggl\|n^{-1}\sum_{j=1}^n 
(\langle X_j,u\rangle^2 - 
{\mathbb E}\langle X,u\rangle^2)\frac{\Sigma u}{\langle \Sigma u,u\rangle}+
n^{-1}\sum_{j=1}^n \langle X_j,u\rangle X_j'\biggr\|
$$
$$
={\mathbb E}{\mathbb E}_u\biggl\|n^{-1}\sum_{j=1}^n 
(\langle X_j,u\rangle^2 - 
{\mathbb E}\langle X,u\rangle^2)\frac{\Sigma u}{\langle \Sigma u,u\rangle}+
n^{-1}\sum_{j=1}^n \langle X_j,u\rangle X_j'\biggr\|
$$
$$
=
{\mathbb E}{\mathbb E}_u\biggl\|n^{-1}\sum_{j=1}^n 
(\langle X_j,u\rangle^2 - 
{\mathbb E}\langle X,u\rangle^2)\frac{\Sigma u}{\langle \Sigma u,u\rangle}+
\biggl(n^{-1}\sum_{j=1}^n \langle X_j,u\rangle^2\biggr)^{1/2}\frac{X'}{\sqrt{n}}
\biggr\|
$$
$$
=
{\mathbb E}\biggl\|n^{-1}\sum_{j=1}^n 
(\langle X_j,u\rangle^2 - 
{\mathbb E}\langle X,u\rangle^2)\frac{\Sigma u}{\langle \Sigma u,u\rangle}+
\biggl(n^{-1}\sum_{j=1}^n \langle X_j,u\rangle^2\biggr)^{1/2}\frac{X'}{\sqrt{n}}
\biggr\|.
$$
Also 
$$
{\mathbb E}\biggl\|n^{-1}\sum_{j=1}^n 
(\langle X_j,u\rangle^2 - 
{\mathbb E}\langle X,u\rangle^2)\frac{\Sigma u}{\langle \Sigma u,u\rangle}+
\biggl(n^{-1}\sum_{j=1}^n \langle X_j,u\rangle^2\biggr)^{1/2}\frac{X'}{\sqrt{n}}
\biggr\|
$$
$$
={\mathbb E}{\mathbb E}'\biggl\|n^{-1}\sum_{j=1}^n 
(\langle X_j,u\rangle^2 - 
{\mathbb E}\langle X,u\rangle^2)\frac{\Sigma u}{\langle \Sigma u,u\rangle}+
\biggl(n^{-1}\sum_{j=1}^n \langle X_j,u\rangle^2\biggr)^{1/2}\frac{X'}{\sqrt{n}}
\biggr\|
$$
$$
\geq {\mathbb E}\biggl\|{\mathbb E}'n^{-1}\sum_{j=1}^n 
(\langle X_j,u\rangle^2 - 
{\mathbb E}\langle X,u\rangle^2)\frac{\Sigma u}{\langle \Sigma u,u\rangle}+
{\mathbb E}'\biggl(n^{-1}\sum_{j=1}^n \langle X_j,u\rangle^2\biggr)^{1/2}\frac{X'}{\sqrt{n}}
\biggr\|
$$
$$
=
{\mathbb E}\biggl(n^{-1}\sum_{j=1}^n \langle X_j,u\rangle^2\biggr)^{1/2}
\frac{{\mathbb E}\|X'\|}{\sqrt{n}}.
$$
Note that 
$${\mathbb E}|\langle X,u\rangle|=\sqrt{\frac{2}{\pi}}\langle \Sigma u,u\rangle^{1/2}.$$
Therefore, 
$$
{\mathbb E}\|X'\|\geq {\mathbb E}\|X\|- {\mathbb E}|\langle X,u\rangle|\frac{\|\Sigma u\|}
{\langle \Sigma u,u\rangle}=
{\mathbb E}\|X\|-\sqrt{\frac{2}{\pi}}\frac{\|\Sigma u\|}{\langle \Sigma u,u\rangle^{1/2}}
$$
and 
$$
{\mathbb E}\biggl\|n^{-1}\sum_{j=1}^n 
(\langle X_j,u\rangle^2 - 
{\mathbb E}\langle X,u\rangle^2)\frac{\Sigma u}{\langle \Sigma u,u\rangle}+
\biggl(n^{-1}\sum_{j=1}^n \langle X_j,u^2\rangle\biggr)^{1/2}\frac{X'}{\sqrt{n}}
\biggr\|
$$
$$
\geq 
\langle \Sigma u,u\rangle^{1/2} {\mathbb E}\biggl(n^{-1}\sum_{j=1}^n Z_j^2\biggr)^{1/2}
\frac{{\mathbb E}\|X\|-\sqrt{\frac{2}{\pi}}\frac{\|\Sigma u\|}{\langle \Sigma u,u\rangle^{1/2}}}{\sqrt{n}},
$$
where 
$$Z_j = \frac{\langle X_j,u\rangle}{\langle \Sigma u,u\rangle^{1/2}}, j=1,\dots, n$$
are i.i.d. standard normal random variables. It is easy to check that 
$$
{\mathbb E}\biggl(n^{-1}\sum_{j=1}^n Z_j^2\biggr)^{1/2}\geq c_2
$$
for a positive numerical constant $c_2,$ implying that 
$$
{\mathbb E}\biggl\|n^{-1}\sum_{j=1}^n 
(\langle X_j,u\rangle^2 - 
{\mathbb E}\langle X,u\rangle^2)\frac{\Sigma u}{\langle \Sigma u,u\rangle}+
\biggl(n^{-1}\sum_{j=1}^n \langle X_j,u\rangle^2\biggr)^{1/2}\frac{X'}{\sqrt{n}}
\biggr\|
$$
$$
\geq 
c_2 
\frac{\langle \Sigma u,u\rangle^{1/2}{\mathbb E}\|X\|-\sqrt{\frac{2}{\pi}}\|\Sigma u\|}{\sqrt{n}}.
$$
We now combine this bound with (\ref{odin}) and (\ref{dva}) to get
$$
{\mathbb E}\|\hat \Sigma-\Sigma\|\geq 
c_2 
\sup_{\|u\|\leq 1}\frac{\langle \Sigma u,u\rangle^{1/2}{\mathbb E}\|X\|-
\sqrt{\frac{2}{\pi}} \|\Sigma u\|}{\sqrt{n}}
$$
$$
\geq 
c_2 
\frac{\|\Sigma\|^{1/2}{\mathbb E}\|X\|-\sqrt{\frac{2}{\pi}} \|\Sigma\|}{\sqrt{n}}
\geq c_2 \|\Sigma\|\biggl(\frac{\sqrt{{\bf r}(\Sigma)}-\sqrt{\frac{2}{\pi}}}{\sqrt{n}}\biggr).
$$
We also have the following obvious bound 
$$
{\mathbb E}\|\hat \Sigma-\Sigma\|\geq \sup_{\|u\|\leq 1}
{\mathbb E}\biggl|n^{-1}\sum_{j=1}^n\langle X_j,u\rangle^2 - 
{\mathbb E}\langle X,u\rangle^2\biggr|
$$
$$
= 
\sup_{\|u\|\leq 1}\langle \Sigma u,u\rangle {\mathbb E}\biggl|n^{-1}\sum_{j=1}^n Z_j^2-1\biggr|
\geq c_3 \frac{\|\Sigma\|}{\sqrt{n}}
$$ 
for some numerical constant $c_3>0.$
Thus, we get 
$$
{\mathbb E}\|\hat \Sigma-\Sigma\|\geq 
c_2 \|\Sigma\|\biggl(\frac{\sqrt{{\bf r}(\Sigma)}-\sqrt{\frac{2}{\pi}}}{\sqrt{n}}\biggr)
\bigvee c_3 \frac{\|\Sigma\|}{\sqrt{n}}
$$
$$
\geq 
\frac{1}{2}\biggl(c_2 \|\Sigma\|\biggl(\frac{\sqrt{{\bf r}(\Sigma)}-\sqrt{\frac{2}{\pi}}}{\sqrt{n}}\biggr)
+c_3 \frac{\|\Sigma\|}{\sqrt{n}}\biggr)
\geq \frac{c_2}{2} \|\Sigma\|\frac{\sqrt{{\bf r}(\Sigma)}}{\sqrt{n}},
$$
provided $c_2$ is chosen to be small enough to satisfy $c_2\sqrt{\frac{2}{\pi}}\leq c_3.$  

This completes the proof in the case when ${\bf r}(\Sigma)\leq 2n$ since in this 
case 
$$
\frac{{\bf r}(\Sigma)}{n}\lessim \sqrt{\frac{{\bf r}(\Sigma)}{n}}.
$$
On the other hand, under the assumption that ${\bf r}(\Sigma)\geq 2n,$ 
\begin{align}
\label{very_easy_b3}
&
{\mathbb E}\|\hat \Sigma-\Sigma\|\geq {\mathbb E}\|\hat \Sigma\|-\|\Sigma\|
\geq 
{\mathbb E}\sup_{\|u\|\leq 1}n^{-1}\sum_{j=1}^{n}\langle X_j,u\rangle^2 -
\|\Sigma\| 
\\
&
\nonumber
\geq 
{\mathbb E}\sup_{\|u\|\leq 1}\frac{\langle X_1,u\rangle^2}{n} -
\|\Sigma\| 
\geq 
\frac{{\mathbb E}\|X\|^2}{n}-\|\Sigma\| 
\\
&
\nonumber
\geq
\frac{({\mathbb E}\|X\|)^2}{n}-\|\Sigma\| 
=\|\Sigma\|\biggl(\frac{{\bf r}(\Sigma)}{n}-1\biggr)\geq \frac{1}{2}\|\Sigma\|\frac{{\bf r}(\Sigma)}{n},
\end{align}
which completes the proof in the case when ${\bf r}(\Sigma)\geq 2n.$

\qed
\end{proof}

Our next goal is to prove a concentration inequality for $\|\hat \Sigma-\Sigma\|$ around its 
median or around its expectation.  In what follows, ${\rm Med}(\xi)$ denotes a median of 
a random variable $\xi.$ 

\begin{theorem}
\label{cor_2}
Let $X,X_1,\ldots,X_n$ be i.i.d. centered Gaussian random vectors in $E$ with  covariance $\Sigma$
and let $M$ be either the median, or the expectation of $\|\hat \Sigma-\Sigma\|.$   
Then, there exists a constant 
$C>0$ such that the folllowing holds. 
If ${\bf r}(\Sigma)\leq n,$ then for all $t\geq 1,$ with probability at least $1-e^{-t},$ 
\begin{align}
\label{sha_sha_conc_B'}
&
\Bigl|\|\hat\Sigma - \Sigma\|-M\Bigr| 
\leq C\|\Sigma\|\biggl(\sqrt{\frac{t}{n}}\bigvee \frac{t}{n}\biggr).
\end{align}
On the other hand, if  ${\bf r}(\Sigma)\geq n,$ then with the same probability 
\begin{align}
\label{sha_sha_conc_B''}
&
\Bigl|\|\hat\Sigma - \Sigma\|-M\Bigr| 
\leq 
C\|\Sigma\|\biggl(\sqrt{\frac{{\bf r}(\Sigma)}{n}}\sqrt{\frac{t}{n}}\bigvee \frac{t}{n}\biggr).
\end{align}
\end{theorem}

In the case when $M$ is the median, this result is an immediate consequence of Theorem \ref{th_operator} and Theorem \ref{spectrum_sharper} that is given below and that provides an equivalent concentration inequality written in a somewhat implicit form. The bounds of 
Theorem \ref{cor_2} in the case when $M$ is the median imply that 
$$
\Bigl|{\mathbb E}\|\hat \Sigma-\Sigma\|- {\rm Med}(\|\hat \Sigma-\Sigma\|)\Bigr|\lesssim 
\|\Sigma\|\frac{1}{\sqrt{n}}
$$
when ${\bf r}(\Sigma)\leq n,$ and 
$$
\Bigl|{\mathbb E}\|\hat \Sigma-\Sigma\|- {\rm Med}(\|\hat \Sigma-\Sigma\|)\Bigr|\lesssim 
\|\Sigma\|\sqrt{\frac{{\bf r}(\Sigma)}{n}}\frac{1}{\sqrt{n}},
$$
when ${\bf r}(\Sigma)\geq n.$ This, in turn, implies the concentration bound in the case
when $M$ is the expectation.

\begin{theorem}
\label{spectrum_sharper} 
Let $X,X_1,\ldots,X_n$ be i.i.d. centered Gaussian random vectors in $E$ with  covariance $\Sigma$
and let $M$ be the median of $\|\hat \Sigma-\Sigma\|.$  
Then, there exists a constant $C>0$ such that for all $t\geq 1$ 
with probability at least $1-e^{-t},$ 
\begin{align}
\label{sha_sha_conc}
&
\Bigl|\|\hat\Sigma - \Sigma\|-M\Bigr| 
\leq C\biggl[\|\Sigma\|\biggl(\sqrt{\frac{t}{n}}\bigvee \frac{t}{n}\biggr)
\bigvee 
\|\Sigma\|^{1/2}M^{1/2}
\sqrt{\frac{t}{n}}
\biggr].
\end{align}
\end{theorem}

The proof of Theorem \ref{spectrum_sharper} is somewhat long and will be given in the next section.
Here we will state a couple corollaries of this theorem.

\begin{corollary}
\label{cor_1}
Under the assumptions and notations of Theorem \ref{spectrum_sharper}, there exists a constant 
$C>0$ such that, for all $t\geq 1,$ with probability at least $1-e^{-t},$ 
\begin{align}
\label{sha_sha_conc_A}
&
\|\hat\Sigma - \Sigma\|\leq 2M
+C\|\Sigma\|\biggl(\sqrt{\frac{t}{n}}\bigvee \frac{t}{n}\biggr).
\end{align}
\end{corollary}

\begin{proof}
The proof easily follows from the next simple bound:
$
2\|\Sigma\|^{1/2}M^{1/2}
\sqrt{\frac{t}{n}}\leq 
M+\|\Sigma\|\frac{t}{n}.
$
\qed
\end{proof}

The following corollary can be viewed as an infinite-dimensional generalization 
of Theorem \ref{classical}.

\begin{corollary}
\label{cor_3}
Under the assumptions and notations of Theorem \ref{spectrum_sharper}, there exists a constant 
$C>0$ such that, for all $t\geq 1,$ with probability at least $1-e^{-t},$ 
 \begin{align}
\label{sha_sha_conc_A'}
&
\|\hat\Sigma - \Sigma\|\leq 
C\|\Sigma\|\biggl(\sqrt{\frac{{\bf r}(\Sigma)}{n}}\bigvee \frac{{\bf r}(\Sigma)}{n}
\bigvee 
\sqrt{\frac{t}{n}}\bigvee \frac{t}{n}\biggr).
\end{align}
This implies that for all $p\geq 1$ 
\begin{equation}
\label{L_p_bound}
\E^{1/p}\|\hat\Sigma - \Sigma\|^p \lesssim_{p}
\|\Sigma\|\max\left\lbrace \sqrt{\frac{\mathbf{r}(\Sigma)}{n}}, \frac{\mathbf{r}(\Sigma)}{n}\right\rbrace.
\end{equation}
\end{corollary}

\begin{proof}
Bound (\ref{sha_sha_conc_A'}) follows immediately from Corollary \ref{cor_1} and 
Theorem \ref{th_operator}. Bound (\ref{L_p_bound}) follows from (\ref{sha_sha_conc_A'})
by integrating the tail probabilities. 

\qed
\end{proof}

\section{Proof of the concentration inequality}

In this section, we provide a proof of Theorem \ref{spectrum_sharper}.
We will use the following well known fact (see, e.g., \cite{Kwapien}).

\begin{theorem}
\label{Kwapien}
Let $X$ be a centered Gaussian random 
variable in a separable Banach space $E.$ Then there exists a sequence $\{x_k: k\geq 1\}$
of vectors in $E$ and a sequence $\{Z_k:k\geq 1\}$ of i.i.d. standard normal random variables 
such that 
$$X=\sum_{k=1}^{\infty} Z_k x_k,$$ 
where the series in the right hand side converges in $E$
a.s. and 
$$
\sum_{k=1}^{\infty}\|x_k\|^2<+\infty.
$$ 
\end{theorem}

Note that under the assumptions and notations of Theorem \ref{Kwapien},
$$
\Sigma u = \sum_{k=1}^{\infty}\langle x_k,u\rangle x_k,\ u\in E^{\ast}.
$$

It easily follows from Theorem \ref{Kwapien} that, for $X^{(m)}:=\sum_{k=1}^m Z_k x_k,$ we have 
$$
{\mathbb E}\|X^{(m)}-X\|^2\to 0\ {\rm as}\ m\to \infty.
$$ 
Let now $\Sigma^{(m)}$ be the covariance operator of $X^{(m)}$ and 
$\hat \Sigma^{(m)}$ be the sample covariance operator based on observations $(X_1^{(m)}, \dots, X_n^{(m)})$ (with the notation $X_j^{(m)}$ having an obvious meaning and the sample 
size $n$ being fixed). Then, 
$$
\|\Sigma^{(m)}-\Sigma\|\to 0 
\ {\rm and} \  
{\mathbb E}\|\hat \Sigma^{(m)}-\hat \Sigma\|\to 
0\ {\rm as}\ m\to \infty.
$$ 
Thus, it is enough to prove the theorem only in the 
case when 
$$
X=X^{(m)}=\sum_{k=1}^m Z_k x_k.
$$ 
The general 
case would then follow by a straightforward limiting argument. 

The main ingredient of the proof is the classical Gaussian concentration inequality
(see, e.g., Ledoux and Talagrand \cite{Ledoux}, p. 21).

\begin{lemma}
\label{Gaussian_concentration}
Let $Z=(Z_1,\dots, Z_N)$ be a standard normal vector in ${\mathbb R}^N$ and 
let $f:{\mathbb R}^N\mapsto {\mathbb R}$ be a function satisfying 
the following Lipschitz condition with some $L>0:$
$$
\Bigl|f(z_1,\dots, z_N)-f(z_1',\dots, z_N')\Bigr|\leq 
L\biggl(\sum_{j=1}^N |z_j-z_j'|^2\biggr)^{1/2},\ z_1,\dots, z_N, z_1',\dots,z_N'\in {\mathbb R}.
$$ 
Then, for all $t>0,$ 
$$
{\mathbb P}\Bigl\{|f(Z)-{\rm Med}(f(Z))|\geq t\Bigr\}\leq 2\biggl(1-\Phi\biggl(\frac{t}{L}\biggr)\biggr),
$$
where $\Phi$ is the distribution function of a standard normal random variable. 
\end{lemma}

This result easily follows from the Gaussian isoperimetric inequality. We will also need another 
consequence of this inequality:

\begin{lemma}
\label{Gaussian_concentration_A}
Under the assumptions of Lemma \ref{Gaussian_concentration}, suppose that for some $M$ and for some  $\alpha>0$
$$
{\mathbb P}\{f(Z)\geq M\}\geq \alpha.
$$
Then, there exists a constant $D>0$ (possibly depending on $\alpha$) such that, for all $t\geq 1,$ with probability 
at least $1-e^{-t},$  
$$f(Z)\geq M - D L\sqrt{t}.$$
\end{lemma}

Denote 
$$
g(X_1,\dots, X_n):=\|W\|\varphi\Bigl(\frac{\|W\|}{\delta}\Bigr),
$$
where
$$
W=\hat \Sigma-\Sigma,\ \ 
\hat \Sigma u = n^{-1}\sum_{j=1}^n \langle X_j,u\rangle X_j,
$$
where $\varphi$ is an arbitrary fixed Lipschitz function with constant $1$ on ${\mathbb R}_{+},$ $0\leq \varphi (s)\leq 1,$
$\varphi (s)=1, s\leq 1,$ $\varphi(s)=0, s>2,$
and where $\delta>0$ is a fixed number (to be chosen later). 
With a little abuse of notation, assume for now that 
$Z:=(Z_{k,j}, k=1,\dots, m, j=1,\dots, n)\in {\mathbb R}^{mn},$
$Z':=(Z_{k,j}', k=1,\dots, m, j=1,\dots, n)\in {\mathbb R}^{mn}$
are nonrandom vectors in ${\mathbb R}^{mn}$
and $X_1,\dots, X_n, X_1',\dots, X_n'$ 
are nonrandom vectors in $E$ defined as follows:
$$
X_j = \sum_{k=1}^m Z_{k,j} x_k,\ X_j' = \sum_{k=1}^m Z_{k,j}' x_k.
$$

Lemma \ref{Gaussian_concentration} will be applied to the function 
$f(Z)=g(X_1,\dots, X_n).$
We have to check the Lipschitz condition for this function.
To this end, we will prove the following lemma.

\begin{lemma}
\label{Lipschitz_constant_XYZ}
There exists a numerical constant $D>0$ such that, for all $Z,Z'\in {\mathbb R}^{mn},$
\begin{equation}
\label{lip_lip_CC_XYZ}
|f(Z)-f(Z')| 
\leq 
D
\frac{\|\Sigma\|+\|\Sigma\|^{1/2}\sqrt{\delta}}{\sqrt{n}}
\biggl(\sum_{j=1}^n \sum_{k=1}^m|Z_{k,j}-Z_{k,j}'|^2\biggr)^{1/2}.
\end{equation}
\end{lemma}

\begin{proof}
Obviously, $0\leq g(X_1,\dots,X_n)\leq 2\delta,$ $0\leq g(X_1',\dots,X_n')\leq 2\delta,$
implying that 
\begin{equation}
\label{fdelta}
|g(X_1,\dots, X_n)-g(X_1',\dots, X_n')|\leq 2\delta.
\end{equation}
Denote 
$$
W'=\hat \Sigma'-\Sigma, \ \ 
\hat \Sigma'=n^{-1}\sum_{j=1}^n X_j'\otimes X_j'.
$$
It is enough to consider the case when $\|W\|\leq 2\delta$ or $\|W'\|\leq 2\delta$ (otherwise, the claim of the lemma is obvious). To be specific, assume that $\|W\|\leq 2\delta.$ Then, using the assumption that $\varphi$ is Lipschitz 
with constant $1,$ we get
\begin{align}
\label{li_li}
&
|g(X_1,\dots, X_n)-g(X_1',\dots, X_n')|=
\biggl|\|W\|\varphi\Bigl(\frac{\|W\|}{\delta}\Bigr)-
\|W'\|\varphi\Bigl(\frac{\|W'\|}{\delta}\Bigr)\biggr|
\\
&
\nonumber
\leq \|W-W'\| + \frac{\|W\|}{\delta}\|W-W'\|
\leq 3\|W-W'\|.
\end{align}

We will now control $\|W-W'\|.$ Note that 
$$
\|W-W'\|= 
\sup_{\|u\|\leq 1,\|v\|\leq 1}
\Bigl|\Bigl\langle (W-W')u,v\Bigr\rangle\Bigr|
$$
$$
=
\sup_{\|u\|\leq 1,\|v\|\leq 1}
\biggl|n^{-1}\sum_{j=1}^n 
\langle X_j,u\rangle \langle X_j,v\rangle
-\langle X_j',u\rangle \langle X_j',v\rangle
\biggr|
$$
$$
\leq 
\sup_{\|u\|\leq 1,\|v\|\leq 1}
\biggl|n^{-1}\sum_{j=1}^n 
\langle X_j,u\rangle \langle X_j-X_j',v\rangle
\biggr|+
\sup_{\|u\|\leq 1,\|v\|\leq 1}
\biggl|n^{-1}\sum_{j=1}^n 
\langle X_j-X_j',u\rangle \langle X_j',v\rangle
\biggr|
$$
$$
\leq 
\sup_{\|u\|\leq 1}
\biggl(n^{-1}\sum_{j=1}^n 
\langle X_j,u\rangle^2
\biggr)^{1/2} 
\sup_{\|v\|\leq 1}
\biggl(n^{-1}\sum_{j=1}^n\langle X_j-X_j',v\rangle^2\biggr)^{1/2}
$$
$$
+
\sup_{\|u\|\leq 1}
\biggl(n^{-1}\sum_{j=1}^n 
\langle X_j-X_j',u\rangle^2
\biggr)^{1/2} 
\sup_{\|v\|\leq 1}
\biggl(n^{-1}\sum_{j=1}^n\langle X_j',v\rangle^2\biggr)^{1/2}
$$
$$
\leq \frac{\|\hat \Sigma\|^{1/2}+
\|\hat \Sigma'\|^{1/2}}{\sqrt{n}}
\sup_{\|u\|\leq 1}
\biggl(\sum_{j=1}^n 
\langle X_j-X_j',u\rangle^2
\biggr)^{1/2} 
$$
Since $\|W\|\leq 2\delta,$  
$$
\|\hat \Sigma\|^{1/2}+\|\hat \Sigma'\|^{1/2}
\leq 
2\|\hat \Sigma\|^{1/2}+\|W-W'\|^{1/2}
\leq 
2\|\Sigma\|^{1/2}+2\sqrt{2\delta}+\|W-W'\|^{1/2}.
$$
Therefore,
$$
\|W-W'\|\leq 
\frac{2\|\Sigma\|^{1/2}+2\sqrt{2\delta}}{\sqrt{n}}
\sup_{\|u\|\leq 1}
\biggl(\sum_{j=1}^n 
\langle X_j-X_j',u\rangle^2
\biggr)^{1/2} 
$$
$$
+
\frac{\|W-W'\|^{1/2}}{\sqrt{n}}
\sup_{\|u\|\leq 1}
\biggl(\sum_{j=1}^n 
\langle X_j-X_j',u\rangle^2
\biggr)^{1/2},
$$
which easily implies 
\begin{equation}
\label{E_E'}
\|W-W'\|\leq 
\frac{4\|\Sigma\|^{1/2}+4\sqrt{2\delta}}{\sqrt{n}}
\sup_{\|u\|\leq 1}
\biggl(\sum_{j=1}^n 
\langle X_j-X_j',u\rangle^2
\biggr)^{1/2} 
\bigvee
\frac{4}{n}
\sup_{\|u\|\leq 1}
\sum_{j=1}^n 
\langle X_j-X_j',u\rangle^2.
\end{equation}
Substituting the last bound in (\ref{li_li}), we get 
\begin{align}
\label{li_li''}
&
|g(X_1,\dots, X_n)-g(X_1',\dots, X_n')|\leq 
\\
\nonumber
&
12\frac{\|\Sigma\|^{1/2}+\sqrt{2\delta}}{\sqrt{n}}
\sup_{\|u\|\leq 1}
\biggl(\sum_{j=1}^n 
\langle X_j-X_j',u\rangle^2
\biggr)^{1/2} 
\bigvee
\frac{12}{n}
\sup_{\|u\|\leq 1}
\sum_{j=1}^n 
\langle X_j-X_j',u\rangle^2.
\end{align}
In view of (\ref{fdelta}), the left hand side is also bounded 
from above by $2\delta,$
which allows one to get from (\ref{li_li''}) that 
\begin{align}
\label{li_li'''}
&
|g(X_1,\dots, X_n)-g(X_1',\dots, X_n')|\leq 
\\
\nonumber
&
12 
\frac{\|\Sigma\|^{1/2}+\sqrt{2\delta}}{\sqrt{n}}
\sup_{\|u\|\leq 1}
\biggl(\sum_{j=1}^n 
\langle X_j-X_j',u\rangle^2
\biggr)^{1/2} 
\bigvee
\biggl(\frac{12}{n}
\sup_{\|u\|\leq 1}
\sum_{j=1}^n 
\langle X_j-X_j',u\rangle^2
\bigwedge 2\delta\biggr).
\end{align}
In the case when 
$$
\sup_{\|u\|\leq 1}
\biggl(\sum_{j=1}^n 
\langle X_j-X_j',u\rangle^2
\biggr)^{1/2} 
\leq 
\sqrt{\frac{\delta n}{6}},
$$
we have 
$$
\frac{12}{n}
\sup_{\|u\|\leq 1}
\sum_{j=1}^n 
\langle X_j-X_j',u\rangle^2
\bigwedge 2\delta\leq 
\frac{12}{\sqrt{6}}\frac{\sqrt{\delta}}{\sqrt{n}}
\sup_{\|u\|\leq 1}
\biggl(\sum_{j=1}^n 
\langle X_j-X_j',u\rangle^2
\biggr)^{1/2}. 
$$
It is also easy to check that the same bound holds in the opposite case, too.
As a consequence, (\ref{li_li'''}) implies that with some numerical constant $D>0,$
\begin{align}
\label{li_li_A'''}
&
|g(X_1,\dots, X_n)-g(X_1',\dots, X_n')|\leq 
\\
\nonumber
&
D 
\frac{\|\Sigma\|^{1/2}+\sqrt{\delta}}{\sqrt{n}}
\sup_{\|u\|\leq 1}
\biggl(\sum_{j=1}^n 
\langle X_j-X_j',u\rangle^2
\biggr)^{1/2}. 
\end{align}
We will now upper bound 
$$
\sup_{\|u\|\leq 1}
\biggl(\sum_{j=1}^n 
\langle X_j-X_j',u\rangle^2
\biggr)^{1/2}. 
$$
Note that 
$$
X_j-X_j' = \sum_{k=1}^m (Z_{k,j}-Z_{k,j}') x_k,
$$
implying that 
$$
\sup_{\|u\|\leq 1}
\biggl(\sum_{j=1}^n 
\langle X_j-X_j',u\rangle^2
\biggr)^{1/2} 
\leq 
\sup_{\|u\|\leq 1}
\biggl(\sum_{j=1}^n 
\biggl(\sum_{k=1}^m(Z_{k,j}-Z_{k,j}')\langle x_k,u\rangle\biggr)^2
\biggr)^{1/2} 
$$
$$
\leq 
\sup_{\|u\|\leq 1}
\biggl(\sum_{j=1}^n\sum_{k=1}^m(Z_{k,j}-Z_{k,j}')^2\sum_{k=1}^m\langle x_k,u\rangle^2
\biggr)^{1/2} 
$$
$$
\leq 
\sup_{\|u\|\leq 1}\biggl(\sum_{k=1}^m\langle x_k,u\rangle^2\biggr)^{1/2}
\biggl(\sum_{j=1}^n\sum_{k=1}^m(Z_{k,j}-Z_{k,j}')^2\biggr)^{1/2} 
$$
$$
=(\sup_{\|u\|\leq 1}\langle \Sigma u,u\rangle)^{1/2}
\biggl(\sum_{j=1}^n\sum_{k=1}^m(Z_{k,j}-Z_{k,j}')^2\biggr)^{1/2}=
\|\Sigma\|^{1/2} \biggl(\sum_{j=1}^n\sum_{k=1}^m(Z_{k,j}-Z_{k,j}')^2\biggr)^{1/2}.
$$
Combining this with bound (\ref{li_li_A'''}) yields
(\ref{lip_lip_CC_XYZ}). 
\qed
\end{proof}

In  what follows, denote 
$$
M:={\rm Med}(\|\hat \Sigma-\Sigma\|)\ {\rm and}\ M_g:={\rm Med}(g(X_1,\dots, X_n)).
$$
It follows from lemmas \ref{Gaussian_concentration} and \ref{Lipschitz_constant_XYZ}
that, for all $t\geq 1$ with probability at least $1-e^{-t},$
\begin{equation}
\label{conc_bd}
\Bigl|g(X_1,\dots, X_n)-M_g\Bigr|
\leq 
D_1 \Bigl(\|\Sigma\|+\|\Sigma\|^{1/2}\sqrt{\delta}\Bigr)\sqrt{\frac{t}{n}},
\end{equation}
where $D_1$ is a numerical constant. We will use this bound to get that, on the event where $\|W\|\leq \delta$
and, at the same time, concentration bound (\ref{conc_bd}) holds, we have 
$$
\|W\|=g(X_1,\dots, X_n)\leq M_g+ D_1 \Bigl(\|\Sigma\|+\|\Sigma\|^{1/2}\sqrt{\delta}\Bigr)\sqrt{\frac{t}{n}}
$$ 
$$
\leq M+ D_1 \Bigl(\|\Sigma\|+\|\Sigma\|^{1/2}\sqrt{\delta}\Bigr)\sqrt{\frac{t}{n}}. 
$$
Denote 
$$
A:= M+D_1 \|\Sigma\| \sqrt{\frac{t}{n}},
\ \ \ 
B:=D_1\|\Sigma\|^{1/2}\sqrt{\frac{t}{n}}.
$$
Then we have
\begin{equation}
\label{main_bou}
{\mathbb P}\biggl\{\delta \geq \|W\| \geq A+B\sqrt{\delta}\biggr\}
\leq e^{-t}.
\end{equation}

We will need the following easy fact.

\begin{lemma}
\label{very_easy}
There exists a constant $D_2>0$ such that
for all $t>0,$ with probability at least $1-e^{-t}$
\begin{equation}
\label{very_easy_b1}
\|\hat \Sigma-\Sigma\|\leq 
D_2 \|\Sigma\|\bigg[{\bf r}(\Sigma)\biggl(\sqrt{\frac{t}{n}}\bigvee \frac{t}{n}\biggr)+{\bf r}(\Sigma)+1\biggr].
\end{equation}
In particular, this implies that, for some constant $D_2>0,$ 
\begin{equation}
\label{very_easy_b2}
M\leq D_2 \|\Sigma\|({\bf r}(\Sigma)+1)
\end{equation}
\end{lemma}

\begin{proof}

To prove (\ref{very_easy_b1}), note that 
$$
\|\hat \Sigma-\Sigma\|\leq \biggl|n^{-1}\sum_{j=1}^n \|X_j\|^2-{\mathbb E}\|X\|^2\biggr|+ {\mathbb E}\|X\|^2 + \|\Sigma\|.
$$
It remains to observe that 
$$
{\mathbb E}\|X\|^2\lesssim \Bigl({\mathbb E}\|X\|\Bigr)^2=
\|\Sigma\|{\bf r}(\Sigma)
$$
and, by Bernstein's inequality for $\psi_1$-random variables, with probability 
at least $1-e^{-t}$
$$ 
\biggl|n^{-1}\sum_{j=1}^n \|X_j\|^2-{\mathbb E}\|X\|^2\biggr|
\lesssim  
\Bigl\|\|X\|^2\Bigr\|_{\psi_1} \biggl(\sqrt{\frac{t}{n}}\bigvee \frac{t}{n}\biggr)
$$
$$
\lesssim \Bigl({\mathbb E}\|X\|\Bigr)^2 \biggl(\sqrt{\frac{t}{n}}\bigvee \frac{t}{n}\biggr)
=\|\Sigma\|{\bf r}(\Sigma)\biggl(\sqrt{\frac{t}{n}}\bigvee \frac{t}{n}\biggr).
$$

The proof of (\ref{very_easy_b2}) immediately follows by taking $t=\log 2.$

\qed
\end{proof}

Denote 
$$
\delta_0:=D_2 \|\Sigma\|\bigg[{\bf r}(\Sigma)\biggl(\sqrt{\frac{t}{n}}\bigvee \frac{t}{n}\biggr)+{\bf r}(\Sigma)+1\biggr].
$$
We will define $\delta_k$ for $k\geq 1$ as follows:
$$
\delta_k= A+B\sqrt{\delta_{k-1}}. 
$$
It is easy to see that $\delta_1\leq \delta_0$ (provided that constant $D_2$ is chosen 
to be sufficiently large). Note also that 
$$
\delta_{k}-\delta_{k+1}= B \Bigl(\sqrt{\delta_{k-1}}-\sqrt{\delta_k}\Bigr).
$$
Thus, by induction, $\delta_k, k\geq 0$ is a nonincreasing sequence. 
In view of definition of $\delta_k,$ it follows from (\ref{main_bou}) that
for all $k\geq 1$  
\begin{equation}
\label{main_bou_k}
{\mathbb P}\biggl\{\delta_{k-1} > \|W\| \geq \delta_k\biggr\}
\leq e^{-t}.
\end{equation}
Also, by Lemma \ref{very_easy},
\begin{equation}
\label{main_bou_0}
{\mathbb P}\biggl\{\|W\| \geq \delta_0\biggr\}
\leq e^{-t}.
\end{equation}

Let 
$$
\bar \delta = \inf_{k\geq 1}\delta_k = \lim_{k\to\infty}\delta_k.
$$
Then 
$$
\bar \delta = A+B\sqrt{\bar \delta}
$$
It is easy to check that 
\begin{equation}
\label{bar_de}
\bar \delta \lesssim (A\vee B^2). 
\end{equation}
In addition,
$$
\delta_k-\bar \delta = B\Bigl(\sqrt{\delta_{k-1}}-\sqrt{\bar \delta}\Bigr)
\leq B\sqrt{\delta_{k-1}-\bar \delta}.
$$
Define $u_k, k\geq 0$ as follows: $u_0=\delta_0,$
$$
u_k = B \sqrt{u_{k-1}}.
$$
Then, 
$$
\delta_k-\bar \delta \leq u_k, k\geq 0.
$$
It is also easy to check that 
$$
u_k = B^{1+2^{-1}+\dots 2^{-{k-1}}}\delta_0^{2^{-k}} = B^2 \biggl(\frac{\delta_0}{B^2}\biggr)^{2^{-k}}
$$
implying 
$$
0\leq \delta_k-\bar \delta \leq B^2 \biggl(\frac{\delta_0}{B^2}\biggr)^{2^{-k}}.
$$
Let 
$$
\bar k:= \min\biggl\{k: \biggl(\frac{\delta_0}{B^2}\biggr)^{2^{-k}}\leq 2\biggr\}.
$$
Clearly, 
$$
\delta_{\bar k}= \bar \delta + \delta_{\bar k}-\bar \delta 
\leq \bar \delta + 2B^2 \lessim A\vee B^2,
$$
where we also used (\ref{bar_de}). Taking into account (\ref{main_bou_k}) and (\ref{main_bou_0}), we get that for some constant $D_3>0$
\begin{equation}
\label{almost_done}
{\mathbb P}\{\|W\|\geq D_3(A\vee B^2)\}\leq 
{\mathbb P}\biggl\{\|W\| \geq \delta_{\bar k}\biggr\}
\leq (\bar k+1) e^{-t}.
\end{equation}
Observe that 
$$
A\vee B^2 \lesssim M\bigvee \|\Sigma\|\biggl(\sqrt{\frac{t}{n}}\bigvee \frac{t}{n}\biggr) 
$$
and also that, for some constant $c_1>0$ and for $t\geq 1$  
$$
\bar k \lesssim \log \log \frac{\delta_0}{B^2}\lessim \log \log (c_1{\bf r}(\Sigma)) \bigvee 
\log \log (c_1 n).
$$
Using now (\ref{almost_done}) with $t+ \log (\bar k+1)$ instead of $t,$
it is easy to get that with probability at least $1-e^{-t}$
$$
\|\hat\Sigma - \Sigma\| \lesssim M\bigvee \|\Sigma\|
\biggl[\sqrt{\frac{t}{n}}\bigvee \frac{t}{n} \bigvee 
$$
$$
\sqrt{\frac{\log^{[3]} (c_1{\bf r}(\Sigma))}{n}}
\bigvee \frac{\log^{[3]} (c_1{\bf r}(\Sigma))}{n}
\bigvee \sqrt{\frac{\log^{[3]} (c_1 n)}{n}}\bigvee \frac{\log^{[3]} (c_1 n)}{n} 
\biggr],
$$ 
where we used the notation 
$\log^{[3]} x := \log\log \log x.$
In the case when ${\bf r}(\Sigma)\lesssim n,$ we have 
$$
\log^{[3]} (c_1{\bf r}(\Sigma))\leq \log^{[3]} (2c_1 n).
$$
Hence, doubling the value of the constant $c_1$ allows us to drop the two terms involving 
$\frac{\log^{[3]} (c_1{\bf r}(\Sigma))}{n}.$ On the other hand, assume that ${\bf r}(\Sigma)\geq C'n$
with a sufficiently large constant $C'$ (to be determined later). 
Observe that $\log^{[3]} (c_1 {\bf r}(\Sigma))\lesssim {\bf r}(\Sigma)$ and we can use  
a bound for the median $M$ similar to (\ref{very_easy_b3}):
\begin{align}
\label{very_easy_b3'}
&
M\geq {\rm Med}(\|\hat \Sigma\|)-\|\Sigma\|
\geq 
{\rm Med}\biggl(\sup_{\|u\|\leq 1}n^{-1}\sum_{j=1}^{n}\langle X_j,u\rangle^2\biggr) -
\|\Sigma\| 
\\
&
\nonumber
\geq 
{\rm Med}\biggl(\sup_{\|u\|\leq 1}\frac{\langle X_1,u\rangle^2}{n}\biggr) -
\|\Sigma\| 
\geq 
\frac{{\rm Med }\|X\|^2}{n}-\|\Sigma\| 
\\
&
\nonumber
=
\frac{({\rm Med}\|X\|)^2}{n}-\|\Sigma\| 
\geq \|\Sigma\|\biggl(\frac{c'{\bf r}(\Sigma)}{n}-1\biggr)\geq \frac{c'}{2}\|\Sigma\|\frac{{\bf r}(\Sigma)}{n},
\end{align}
for some constants $c'>0$ and for $C'\geq 2/c'.$ We also used the fact that for Gaussian $X$
$$
{\rm Med}\|X\|\asymp {\mathbb E}\|X\| = \Bigl(\|\Sigma\|{\bf r}(\Sigma)\Bigr)^{1/2}.
$$
Thus, we get  
$$ 
\|\Sigma\|\biggl(\sqrt{\frac{\log^{[3]} (c_1{\bf r}(\Sigma))}{n}}
\bigvee \frac{\log^{[3]} (c_1{\bf r}(\Sigma))}{n}\biggr)\lesssim 
\|\Sigma\|\frac{{\bf r}(\Sigma)}{n}\lesssim M.
$$ 
Since also $\frac{\log^{[3]} (c_1 n)}{n}\lesssim 1,$ this implies 
that with some constant $C_1$ and with the same probability  
\begin{align}
\label{sha_sha''}
&
\nonumber 
\|W\|=\|\hat\Sigma - \Sigma\| 
\\
&
\leq C_1\biggl[M\bigvee  \|\Sigma\|
\biggl(\sqrt{\frac{t}{n}}\bigvee \frac{t}{n} 
\bigvee \sqrt{\frac{\log^{[3]} (c_1 n)}{n}}\biggr)\biggr].
\end{align}
Take now $\delta$ to be equal to the expression in the right hand side of bound (\ref{sha_sha''}) and use this value of $\delta$ to do another iteration of 
bound (\ref{main_bou}). This easily yields that with some constant $C>0$ and with probability at least $1-2e^{-t}$
\begin{align}
\label{sha_sha''''}
&
\|W\|=\|\hat\Sigma - \Sigma\| 
\leq 
C\biggl[M\bigvee  \|\Sigma\|
\biggl(\sqrt{\frac{t}{n}}\bigvee \frac{t}{n}\biggr)\biggr].
\end{align}


To complete the proof of concentration inequality (\ref{sha_sha_conc}), note that, for an arbitrary $\delta>0,$ 
on the event where (\ref{conc_bd}) holds and also $\|W\|\leq \delta,$ 
\begin{align}
&
\nonumber
\|\hat\Sigma - \Sigma\|
=
g(X_1,\dots, X_n)
\\
&
\nonumber
\leq 
M_g+D_1 \Bigl(\|\Sigma\|+\|\Sigma\|^{1/2}\sqrt{\delta}\Bigr)\sqrt{\frac{t}{n}}
\\
&
\nonumber
\leq
M+D_1 \Bigl(\|\Sigma\|+\|\Sigma\|^{1/2}\sqrt{\delta}\Bigr)\sqrt{\frac{t}{n}}.
\end{align}
This bound will be used with 
\begin{equation}
\label{del_de}
\delta:=
C\biggl[M\bigvee  \|\Sigma\|
\biggl(\sqrt{\frac{t+2n}{n}}\bigvee \frac{t+2n}{n}\biggr)\biggr].
\end{equation}
Then, in view of bound (\ref{sha_sha''''}), 
$$
{\mathbb P}\{\|W\|\geq \delta\}\leq 2e^{-t-2n}=2 e^{-2n-t} \leq e^{-2n}
$$
(provided that $t\geq 1$).
Note also that 
\begin{align}
&
\nonumber
{\mathbb P}\{g(X_1,\dots, X_n)\geq M\} \geq 
{\mathbb P}\{g(X_1,\dots, X_n)\geq M, \|W\|\leq \delta\} 
\\
&
\nonumber
\geq 
{\mathbb P}\{\|W\|\geq M, \|W\|\leq \delta\} \geq
1/2-{\mathbb P}\{\|W\|>\delta\}\geq 1/2-e^{-2n}\geq 1/4. 
\end{align}
Then, it follows from Lemma \ref{Gaussian_concentration_A}  
that, for a sufficiently large constant $D_1$  and for all $t\geq 1,$ with probability at least 
$1-e^{-t},$
the following bound holds:
$$
g(X_1,\dots, X_n) \geq M- D_1 \Bigl(\|\Sigma\|+\|\Sigma\|^{1/2}\sqrt{\delta}\Bigr)\sqrt{\frac{t}{n}}.
$$
Recall also that $g(X_1,\dots,X_n)=\|W\|$ on the event where $\|W\|\leq \delta$ of 
probability at least $1-2e^{-t-2n}\geq 1-e^{-t}.$ 
Therefore, with probability at least $1-3e^{-t},$
\begin{align}
\label{last_bd}
&
\Bigl|\|\hat\Sigma - \Sigma\|-M\Bigr|
\leq 
D_1 \Bigl(\|\Sigma\|+\|\Sigma\|^{1/2}\sqrt{\delta}\Bigr)\sqrt{\frac{t}{n}}
\end{align}
The result now follows by substituting $\delta$ given by (\ref{del_de}) into bound (\ref{last_bd}),  
doing simple algebra and adjusting the value of constant $D_1$ to get the probability bound 
$1-e^{-t}.$

\qed 

Very recent exponential generic chaining bounds for empirical processes by Dirksen \cite{Dirksen} (see Corollary 5.7) and by Bednorz\cite{Bednorz} (see Theorem 1) imply the following (earlier, Mendelson \cite{Mendelson-1}, Theorem 3.1 obtained another version of exponential generic chaining bounds 
for the same class of processes). 

\begin{theorem}
\label{dirksen}
Let $X,X_1,\dots, X_n$ be i.i.d. random variables in a measurable space $(S,{\mathcal A})$ with common distribution 
$P$ and let ${\cal F}$ be a class of measurable functions on $(S,{\mathcal A}).$
There exists a constant $C>0$ such that for all $t\geq 1$ with probability at 
least $1-e^{-t}$ 
\begin{align}
&
\nonumber
\sup_{f\in {\mathcal F}} 
\left| \frac{1}{n}\sum_{i=1}^n f^2(X_i)-   \mathbb E f^2(X) \right|
\\
&
\nonumber
\leq C\max\left\lbrace\sup_{f\in {\mathcal F}}\|f\|_{\psi_2} 
\frac{\gamma_2({\mathcal F};\psi_2)}{\sqrt{n}},
\frac{\gamma_2^2({\mathcal F};\psi_2)}{n}, 
\sup_{f\in {\mathcal F}}\|f\|_{\psi_2}^2\sqrt{\frac{t}{n}}, \sup_{f\in {\mathcal F}}\|f\|_{\psi_2}^2\frac{t}{n}\right\rbrace.
\end{align}
\end{theorem}

This result together with the argument used in the proof of the upper bound of Theorem \ref{th_operator} easily implies the following generalization of Corollary \ref{cor_3}.

\begin{theorem}
\label{subgau}
Let $X,X_1,\ldots,X_n$ be i.i.d. weakly square integrable centered random vectors in $E$ with covariance operator $\Sigma.$ If $X$ is subgaussian and pregaussian, then
there exists a constant 
$C>0$ such that, for all $t\geq 1,$ with probability at least $1-e^{-t},$ 
 \begin{align}
\label{sha_sha_conc_AB}
&
\|\hat\Sigma - \Sigma\|\leq 
C\|\Sigma\|\biggl(\sqrt{\frac{{\bf r}(\Sigma)}{n}}\bigvee \frac{{\bf r}(\Sigma)}{n}
\bigvee 
\sqrt{\frac{t}{n}}\bigvee \frac{t}{n}\biggr).
\end{align}
\end{theorem}

Note that the proof of concentration inequality of Theorem \ref{spectrum_sharper} does not rely on generic chaining bounds, it relies only on the Gaussian isoperimetric inequality. The bound of Theorem 
\ref{subgau} (based on the generic chaining method) could be used to provide a shortcut in the proof 
of the concentration inequality. To this end, instead of using very rough initial bound $\delta_0$ based on Lemma \ref{very_easy} one should use much more precise bound of Theorem \ref{subgau}. 
In this case, there is no need to implement an iterative argument improving the bound, the concentration inequality in its explicit form (Theorem \ref{cor_2}) follows just by an application of the Gaussian isoperimetric inequality. Adamczak \cite{Adamczak} suggested an alternative approach to the proof of Theorem \ref{cor_2}. It is based on a version of a concentration inequality for Gaussian chaos 
and on some other tools (such as Gordon-Chevet inequality), but it does not rely on the generic 
chaining bounds.

{\bf Acknowledgments.} The authors are very thankful to Sjoerd Dirksen for attracting their attention 
to paper \cite{Dirksen}. Radek Adamczak pointed out that a similar result was proved in \cite{Bednorz}.

The authors are especially thankful to Radek Adamczak for providing an alternative proof of the concentration inequality and for very helpful discussions. The initial version of Theorem \ref{cor_2}
was under an extra assumption that ${\bf r}(\Sigma)\lesssim e^{2n}.$ We improved our argument 
after Adamczak had provided his alternative proof.

\end{document}